\documentclass[9pt]{article}
\usepackage{amsmath,amsthm}
\usepackage{amssymb}
\usepackage{amsfonts}
\usepackage{latexsym}
\usepackage{graphics,graphicx}

\allowdisplaybreaks

\title{Existence, nonexistence and multiplicity of positive solutions for singular boundary value problems involving $\varphi$-Laplacian}

\begin{document}
\author{Chan-Gyun Kim 
 \\
 {\small Department of Mathematics Education, Pusan National University, \hfill{\ }}\\
\ \ {\small \qquad \ \ Busan, 609-735, Korea\hfill {\ }}\\
 }

\date{}
\maketitle

\begin{abstract}
 In this paper, we establish the results on the existence, nonexistence and multiplicity of positive solutions to singular boundary value problems involving $\varphi$-Laplacian. Our approach is based on the fixed point index theory. The interesting point is that a result for the existence of three positive solutions is given.
\end{abstract}



\newtheorem{Thm}{Theorem}[section]
\newtheorem{Lem}[Thm]{Lemma}
\newtheorem{Cor}[Thm]{Corollary}
\newtheorem{Def}[Thm]{Definition}
\newtheorem{Ex}[Thm]{Example}
\newtheorem{Rem}[Thm]{Remark}
\newtheorem{Conj}[Thm]{Conjecture}
\newtheorem{Pro}[Thm]{Proposition}
\numberwithin{equation}{section}

\section{Introduction}

In this paper, we study the existence, nonexistence and multiplicity of positive solutions to the following problem
\begin{equation}\label{1.1}
\begin{cases}
(d(t)\varphi(c(t)u'))'+\lambda h(t)f(u)=0,~~t \in (0,1),\\
u(0)=u(1)=0,
\end{cases}
\end{equation}
where $\varphi:\mathbb{R} \to \mathbb{R}$ is an odd increasing homeomorphism, $c,d \in C([0,1],(0,\infty))$, $\lambda \in \mathbb{R}_+:=[0,\infty)$ is a parameter, $h \in C((0,1),\mathbb{R}_+)\setminus \{0\}$ and $f \in C(\mathbb{R}_+,\mathbb{R}_+)$ with $f(s)>0$ for $s >0$.

Throughout this paper, the homeomorphism $\varphi$ satisfies the following assumption:
\begin{itemize}
\item[$(A)$]
there exist increasing homeomorphisms $\psi_1,\psi_2: \mathbb{R}_+\to \mathbb{R}_+$ such that
    \begin{center}
    $\varphi(x)\psi_1(y)\le \varphi(xy) \le \varphi(x) \psi_2(y)~\hbox{for all }~x,y \in \mathbb{R}_+.$
    \end{center}
\end{itemize}
For convenience, we denote by $\mathcal{H}_\xi$ the set
$$\left\{g \in C((0,1),\mathbb{R}_+): \int_0^{\frac{1}{2}}\xi^{-1}\left(\int_s^{\frac{1}{2}} g(\tau)d\tau\right)ds+\int_{\frac{1}{2}}^1\xi^{-1}\left(\int_{\frac{1}{2}}^s g(\tau)d\tau\right)ds<\infty\right\},$$
where $\xi : \mathbb{R}_+ \to \mathbb{R}_+$ is an increasing homeomorphism, and we make the following notations:
\begin{center}
$\displaystyle f_0:=\lim_{s \to 0^+}\frac{f(s)}{\varphi(s)}~\hbox{and}~f_\infty:=\lim_{s \to \infty}\frac{f(s)}{\varphi(s)}.$
\end{center}
It is well known that \begin{equation}\label{inequality2}
\varphi^{-1}(x)\psi_2^{-1}(y) \le \varphi^{-1}(xy)\le  \varphi^{-1}(x)\psi_1^{-1}(y)~\hbox{for all}~x,y \in \mathbb{R}_+
\end{equation}
and $L^1(0,1)\cap C(0,1) \subsetneq \mathcal{H}_{\psi_1} \subseteq \mathcal{H}_\varphi \subseteq \mathcal{H}_{\psi_2}$ (see, e.g., \cite[Remark 1]{jeong2019existence}).

Problem \eqref{1.1} arises naturally in studying radial solutions to the following quasilinear elliptic equation defined on an annular domain
\begin{equation}\label{0}
\begin{cases}
\mathrm{div}(w(|x|)A(|\nabla v|)\nabla v)+\lambda k(|x|)f(v)=0~~\hbox{in}~ \Omega,\\
v=0~\hbox{on}~\partial \Omega,
\end{cases}
\end{equation}
where $\Omega=\{x \in \mathbb{R}^N: R_1<|x|<R_2\}$ with $N\ge 2$ and $0<R_1<R_2< \infty,$ $w \in C([R_1,R_2],(0,\infty))$ and $k \in C((R_1,R_2),\mathbb{R}_+)$. Indeed, applying change of variables
 \begin{center}
 $v(x)=u(t)$ and $|x|=(R_2-R_1)t+R_1$,
\end{center}
problem \eqref{0} is transformed into problem \eqref{1.1} with $$\varphi(s)=A(|s|)s, ~d(t)=w((R_2-R_1)t+R_1)((R_2-R_1)t+R_1)^{N-1},~c(t)=\frac{1}{R_2-R_1}$$ and $$h(t)=(R_2-R_1)((R_2-R_1)t+R_1)^{N-1}k((R_2-R_1)t+R_1)$$ (see, e.g., \cite{jeong2019existence,MR1946560}).

For $\varphi(s)=|s|^{p-2}s$ with $p>1$, problem \eqref{1.1} has been extensively studied in the literature (see \cite{MR903391,MR1894171,MR1440355,MR3231529,MR1825983,MR1643199,MR1064016} for $p=2$ and \cite{MR1880513,MR1003248,MR2514757,MR2847467,MR3911665,MR3361694,MR1423340,MR1400567,yang2018positive,MR3543777,MR3959306} for $p>1$). For example, when $h\in \mathcal{H}_{\varphi}$ and $c\equiv d \equiv 1$, Agarwal, L\"u and O'Regan \cite{MR1880513} studied the existence and multiplicity of positive solutions to problem \eqref{1.1} under various assumptions on $f_0$ and $f_\infty$. In \cite{MR2514757}, when $h\in \mathcal{H}_{\varphi},$ $c\equiv d \equiv 1$ and $f(s)$ satisfies $f(0)>0$ and $f_\infty=\infty$, it was shown that there exists $\lambda_*>0$ such that \eqref{1.1} has at least two positive solutions for $\lambda \in (0,\lambda_*)$, one positive solution for $\lambda =\lambda_*$ and no positive solution for $\lambda>\lambda_*$. Recently,  under the assumptions that $h \in C^1((0,1],(0,\infty))$ is strictly decreasing, $h(t) \le C t^{-\eta}$ for some $C>0$ and $\eta \in (0,1)$, $c\equiv d \equiv 1$, $f \in C(\mathbb{R}_+,\mathbb{R})$ is differentiable on $(0,\infty)$, $f(0)<0,$ $\displaystyle \limsup_{s \to 0^+} sf'(s) <\infty,$ $f'(s)>0$ for $s>0,$ $\displaystyle \lim_{s \to \infty} f(s)=\infty,$ $\displaystyle \lim_{s \to \infty} \frac{f(s)}{s^{p-1}}=0$ and $\displaystyle\frac{f(s)}{s^q}$ is nondecreasing on $[a,\infty)$ for some $a>0$ and $q \in (0,p-1),$ Shivaji, Sim and Son \cite{MR3543777} showed the uniqueness of positive solution to problem \eqref{1.1} for all large $\lambda>0$. For sign-changing weight $h$ satisfying $|h| \in \mathcal{H}_{\varphi}$ and $c\equiv d \equiv 1$, Xu \cite{MR3959306} studied the existence of a nontrivial solution to problem \eqref{1.1} for all small $\lambda>0$ under the assumptions that $f\in C(\mathbb{R},\mathbb{R})$ is nondecreasing,  $f(0)>0$ and $f_\infty \in (0,\infty)$.

For general $\varphi$ satisfying $(A)$, when $c\equiv d \equiv 1$ and $h\in L^1(0,1) \cap C(0,1)$, Bai and Chen \cite{MR3001526} studied the existence of at least two positive solutions for $\lambda$ belonging to an explicit interval under some assumptions on $f$ satisfying $f(0)=0$. When $c\equiv d \equiv 1$, $h \in \mathcal{H}_{\psi_1}$ and either $f_0=f_\infty=\infty$ or $f_0=f_\infty=0,$ Lee and Xu \cite{not-published} showed that there exist $\lambda^*\ge \lambda_*>0$ such that \eqref{1.1} has at least two positive solutions for $\lambda \in (0,\lambda_*)$, one positive solution for $\lambda \in [\lambda_*,\lambda^*]$ and no positive solution for $\lambda>\lambda_*$. In \cite{jeong2019existence}, for nonnegative nonlinearity $f=f(t,s)$ satisfying $f(t_0,0)>0$ for some $t_0\in [0,1]$ and $h \in \mathcal{H}_\varphi$, the existence of an unbounded solution component was shown and, under several assumptions on $f$ at $\infty$, the existence, nonexistence and multiplicity of positive solutions were studied.

For more general $\varphi$ which does not satisfy $(A)$, when $c\equiv d\equiv 1$ and $0\le h \in L^1(0,1)$ with $h \not \equiv 0$, Kaufmann and Milne \cite{MR3759527} proved the existence of positive solution to problem \eqref{1.1} for all $\lambda>0$ under the assumptions on $f$ which induces the sublinear nonlinearity provided $\varphi(s)=|s|^{p-1}s$ with $p>1$. For other interesting results, we refer the reader to \cite{MR2393405,MR2502700,MR3850394} and the references therein.

The concavity of solutions plays a crucial role in defining a suitable cone so that the solution operator is well defined (see, e.g., \cite{MR1880513,MR2514757,not-published} and the references therein). When $c \equiv d \equiv 1$, it is easy to see that solutions to problem \eqref{1.1} are concave functions on $[0,1]$. However, it is not clear that the solutions are concave functions on $[0,1]$, unless $c \equiv d \equiv 1$. In order to overcome this difficulty, a lemma (\cite[Lemma 2.4]{MR1946560}) was proved, so that various results for the existence, nonexistence and multiplicity of positive solutions to problem \eqref{1.1} were proved in \cite{MR1946560} when $d$ is nondecreasing on $[0,1]$ and $h \in C[0,1]$ satisfies $h \not \equiv 0$ on any subinterval of $[0,1].$

The aim of this paper is to improve on the results in \cite{MR1946560} by assuming the weaker hypotheses on $h$ and $d$ than those in \cite{MR1946560}. More precisely, the monotonicity of $d$ is not assumed, and the weight function $h$ may not be $L^1(0,1)$ and it can be vanished in some subinterval of $(0,1).$ Furthermore, a result for the existence of three positive solutions is given, which does not appear in \cite{MR1946560}.

The rest of this paper is organized as follows. In Section 2, we establish preliminaries which are essential for proving our results in this paper. In Section 3, the main results are proved and an example to illustrate the results obtained in this paper is provided. Finally, the summary of this paper and future work are given in Section 4.

\section{Preliminaries}

In this section we give preliminaries which are essential for proving our results in this paper.

First, we introduce a solution operator related to problem \eqref{1.1}. Let $g \in \mathcal{H}_{\varphi}\setminus\{0\}$ be given, and define a function $\nu_g:(0,1) \to \mathbb{R}$ by
 \begin{center}
 $\nu_g(t)=\nu_g^1(t)-\nu_g^2(t)$ for $t \in (0,1)$.
\end{center}
Here $\nu_g^1$ and $\nu_g^2$ are the functions defined by, for $t\in (0,1),$
$$\displaystyle \nu_g^1(t)=\int_0^t \frac{1}{c(s)}\varphi^{-1} \left(\frac{1}{d(s)}\int_s^t  g(\tau)d\tau\right)ds$$ and $$\displaystyle \nu_g^2(t)=\int_t^1 \frac{1}{c(s)} \varphi^{-1} \left(\frac{1}{d(s)}\int_t^s g(\tau)d\tau\right)ds.$$
It is easy to see that $\nu_g^1$ and $\nu_g^2$ are continuous and monotone functions on $(0,1)$ satisfying $$\displaystyle \lim_{t\to 0+}\nu_g^1(t)=\lim_{t\to 1-}\nu_g^2(t)=0~\hbox{and} ~\displaystyle\lim_{t\to 1-}\nu_g^1(t)=\lim_{t\to 0+}\nu_g^2(t)\in (0,\infty].$$
 Thus there exists an interval $[\sigma_g^1,\sigma_g^2] \subsetneq (0,1)$ satisfying $\nu_g(\sigma)=0$ for all $\sigma \in [\sigma_g^1,\sigma_g^2]$ (see \cite{jeong2019existence}).

Define a function $T : \mathcal{H}_{\varphi} \to C[0,1]$ by $T(0)=0$ and, for $g \in \mathcal{H}_{\varphi}\setminus \{0\}$,
\begin{equation}\label{operator1}
 T(g)(t)=\left\{
            \begin{array}{ll}
             \int_0^t \frac{1}{c(s)}\varphi^{-1} \left(\frac{1}{d(s)}\int_s^{\sigma} g(\tau) d\tau\right)ds, & \hbox{if}~0\le t \le \sigma, \\
              \int_t^1 \frac{1}{c(s)}\varphi^{-1} \left(\frac{1}{d(s)}\int_{\sigma}^s g(\tau) d\tau\right)ds, & \hbox{if}~\sigma \le t \le 1.
            \end{array}
          \right.
\end{equation}
where $\sigma=\sigma(g)$ is a zero of $\nu_g$ in $(0,1)$, i.e.,
\begin{equation}\label{maximum-linear}
\int_0^{\sigma} \frac{1}{c(s)}\varphi^{-1} \left(\frac{1}{d(s)}\int_s^{\sigma} g(\tau)d\tau\right)ds=\int_\sigma^1 \frac{1}{c(s)}\varphi^{-1} \left(\frac{1}{d(s)}\int_{\sigma}^s g(\tau)d\tau\right)ds.
\end{equation}
We notice that, although $\sigma=\sigma(g)$ is not necessarily unique, the right hand side of the equality in \eqref{operator1} does not depend on a particular choice of $\sigma$. In other words, $T(g)$ is independent of the choice of $\sigma \in [\sigma_g^1,\sigma_g^2]$ (see, e.g.,  \cite{jeong2019existence} or \cite{MR1946560}).

For $g \in \mathcal{H}_\varphi,$ consider the following problem
\begin{equation}\label{2.1}
\begin{cases}
-(d(t)\varphi(c(t)u'))'=g(t),~~t \in (0,1),\\
u(0)=u(1)=0.
\end{cases}
\end{equation}
For $g=0,$ \eqref{2.1} has a unique zero solution due to the boundary conditions.

The usual maximum norm in a Banach space $C[0,1]$ is denoted by
 \begin{center}
 $\displaystyle\|u\|_\infty:=\max_{t\in [0,1]} |u(t)|$ for $u \in C[0,1]$
 \end{center}
 and let
$$\displaystyle \rho_1:=\frac{c_0}{\|c\|_\infty}\frac{\psi_2^{-1}\left(\frac{1}{\|d\|_\infty}\right)}{\psi_1^{-1}\left(\frac{1}{d_0}\right)}\in (0,1],$$
where $\displaystyle c_0:=\min_{t\in [0,1]}c(t)>0$ and $\displaystyle d_0:=\min_{t\in [0,1]}d(t)>0$. Recently,  without the monotonicity of $d$, a result similar to \cite[Lemma 2.4]{MR1946560} was proved in \cite{jeong2019existence}.

\begin{Lem}\label{nonconcavity2} \textnormal{(\cite[Lemma 2]{jeong2019existence})}
 Assume that $(A)$ and $g\in \mathcal{H}_\varphi$ hold.  Then $T(g)$ is the unique solution to problem \eqref{2.1} and
 \begin{center}
 $T(g)(t) \ge \min\{t, 1-t\} \rho_1\|T(g)\|_\infty$ for $t \in [0,1].$
 \end{center}
\end{Lem}

From now on, we assume $h \in \mathcal{H}_\varphi\setminus \{0\}$. Let \begin{center}
$\displaystyle\mathcal{A}_h:=\{x \in (0,1) : h(y)=0$ for all $y \in (0,x)\}$ and $\displaystyle\mathcal{B}_h:=\{x \in (0,1) : h(y)=0$ for all $y \in (x,1)\}$.
\end{center}
For convenience, we use the following notations:

\begin{center}
$\displaystyle\alpha_h:=\max \mathcal{A}_h$ if $\mathcal{A}_h \neq \emptyset$ and $\alpha_h:=0$ if $\mathcal{A}_g =\emptyset,$ where $\displaystyle\mathcal{A}_h=\{x \in (0,1) : h(y)=0$ for all $y \in (0,x)\}$;
\end{center}
\begin{center}
$\displaystyle\beta_h:=\min \mathcal{B}_h$ if $\mathcal{B}_h \neq \emptyset$ and $\beta_h:=1$ if $\mathcal{B}_g =\emptyset,$ where $\displaystyle\mathcal{B}_h=\{x \in (0,1) : h(y)=0$ for all $y \in (x,1)\}$;
\end{center}
\begin{center}
$\displaystyle \bar\alpha_h:=\max\{x \in (0,1] : h(y)>0$ for all $y \in (\alpha_h,x)\}$;
\end{center}
\begin{center}
 $\displaystyle \bar\beta_h:=\max\{x \in [0,1) : h(y)>0$ for all $y \in (x,\beta_h)\}$;
\end{center}
\begin{center}
$\displaystyle\gamma_h^1:=4^{-1}(3\alpha_h+\bar\alpha_h)$ and $\displaystyle \gamma_h^2:=4^{-1}(\bar\beta_h+3\beta_h)$.
\end{center}
Since $h \not \equiv 0,$ it follows that
\begin{center}
$\displaystyle 0 \le \alpha_h<\gamma_h^1<\gamma_h^2 <\beta_h \le 1$.
\end{center}

Let $\mathcal{K}:=\{ u \in C([0,1],\mathbb{R}_+) : u(t) \ge \rho_h \|u\|_\infty~\hbox{for}~t \in [\gamma_h^1,\gamma_h^2] \}.$
Here
\begin{center}
$\displaystyle \rho_h:=\rho_1\min\{\gamma_h^1,1-\gamma_h^2\}\in (0,1)$.
\end{center}
Then $\mathcal{K}$ is a cone in $C[0,1].$  For $r>0,$ let $\mathcal{K}_r:=\{u \in \mathcal K~:~\|u\|_\infty<r\}$, $\partial \mathcal{K}_r:=\{u \in \mathcal K~:~\|u\|_\infty=r\}$ and $\overline{\mathcal{K}}_r:=\mathcal{K}_r \cup \partial \mathcal{K}_r$.

Define a function $F: \mathbb{R}_+\times \mathcal{K} \to C(0,1)$ by
 \begin{center}
 $F(\lambda,u)(t)=\lambda h(t)f(u(t))$ for $(\lambda,u) \in \mathbb{R}_+\times \mathcal{K}$ and $t \in (0,1).$
\end{center}
Clearly, $F(\lambda,u) \in \mathcal{H}_\varphi$ for any $(\lambda,u) \in \mathbb{R}_+\times \mathcal{K}$.

Now we define an operator $H : \mathbb{R}_+ \times \mathcal{K} \to C[0,1]$ by
\begin{center}
$H(\lambda,u)\equiv T(F(\lambda,u))$ for $(\lambda,u) \in \mathbb{R}_+\times \mathcal{K}$.
\end{center}
 That is, for $(\lambda,u) \in \mathbb{R}_+\times \mathcal{K},$
\begin{equation*}\label{operator2}
H(\lambda,u)(t)=\left\{
            \begin{array}{ll}
             \int_0^t \frac{1}{c(s)}\varphi^{-1} (\frac{1}{d(s)}\int_s^{\sigma} F(\lambda,u)(\tau)d\tau)ds, & \hbox{if}~0\le t \le \sigma, \\
              \int_t^1 \frac{1}{c(s)}\varphi^{-1} (\frac{1}{d(s)}\int_{\sigma}^s F(\lambda,u)(\tau) d\tau)ds, & \hbox{if}~\sigma \le t \le 1,
            \end{array}
          \right.
\end{equation*}
where $\sigma=\sigma(\lambda,u)$ is a number satisfying
\begin{equation}\label{maximum}
\int_0^{\sigma} \frac{1}{c(s)}\varphi^{-1} \left(\frac{1}{d(s)}\int_s^{\sigma} F(\lambda,u)(\tau)d\tau\right)ds=\int_\sigma^1 \frac{1}{c(s)}\varphi^{-1} \left(\frac{1}{d(s)}\int_{\sigma}^s F(\lambda,u)(\tau)d\tau\right)ds.
\end{equation}

\begin{Rem}\label{solution_property}
\begin{itemize}
\item[$(1)$] For any $(\lambda,u) \in \mathbb{R}_+ \times \mathcal{K},$ by Lemma \ref{nonconcavity2}, $H (\mathbb{R}_+ \times \mathcal{K}) \subseteq \mathcal{K}$.

    \item[$(2)$] It is easy to see that \eqref{1.1} has a solution if and only if $H(\lambda,\cdot)$ has a fixed point in $\mathcal{K}.$
\item[$(3)$] Since $H(0,u) = 0$ for all $u \in \mathcal{K}$, $0$ is a unique solution to problem \eqref{1.1} with $\lambda=0$.
By Lemma \ref{nonconcavity2}, any nonzero solution $u$ to problem \eqref{1.1} is positive one, i.e., $u(t)>0$ for all $t \in (0,1).$
\end{itemize}
    \end{Rem}

By the argument similar to those in the proof of \cite[Lemma 3]{MR1880513}, it can be proved that $H$ is completely continuous on $\mathbb{R}_+\times \mathcal{K}$ (see also \cite[Lemma 3.3]{MR2417920}). We omit the proof of it.

\begin{Lem}\label{complete continuity2} $\textnormal{(\cite[Lemma 4]{jeong2019existence})}$
 Assume that $(A)$ and $g\in \mathcal{H}_\varphi$ hold. Then the operator $H: \mathbb{R}_+\times \mathcal{K} \to \mathcal{K}$ is completely continuous, i.e., compact and continuous.
\end{Lem}

Finally, we recall a well-known theorem of the fixed point index theory.

\begin{Thm}\label{thm_cone} $\textnormal{(see, e.g., \cite{MR787404,MR959889})}$  Assume that, for some $r>0,$ $T_1:\overline{\mathcal{K}}_r \to
\mathcal{K}$ is completely continuous on $\overline{\mathcal{K}}_r$. Then
\begin{center}
$(i)$ if $\|T_1(u)\|_\infty > \|u\|_\infty$ for $u
\in \partial \mathcal K_r$, then $i(T_1,\mathcal K_r,\mathcal{K})=0;$
\end{center}
\begin{center}
 $(ii)$ if $\|T_1(u)\|_\infty < \|u\|_\infty$ for $u
\in \partial \mathcal K_r$, then $i(T_1,\mathcal K_r,\mathcal{K})=1.$
\end{center}

\end{Thm}

\section{Main Results}

Throughout this section, we assume $h \in \mathcal{H}_{\psi_1}\setminus \{0\}$. For convenience, we use the following notations in this section:
\begin{center}
$\displaystyle \gamma_h:=\frac{\gamma_h^1+\gamma_h^2}{2}$,
where $\displaystyle\gamma_h^1=\frac{3\alpha_h+\bar \alpha_h}{4}$ and $\displaystyle \gamma_h^2=\frac{\bar\beta_h+3\beta_h}{4};$
\end{center}
\begin{center}
$\displaystyle A_1:=\frac{1}{\|c\|_\infty}\psi_2^{-1}(\frac{1}{\|d\|_\infty})\min \left\{\int_{\gamma_h^1}^{\gamma_h}\psi_2^{-1}\left(\int_s^{\gamma_h}h(\tau)d\tau\right)ds, \int_{\gamma_h}^{\gamma_h^2}\psi_2^{-1}\left(\int^s_{\gamma_h}h(\tau)d\tau\right)ds\right\};$
\end{center}
 \begin{center}
 $\displaystyle A_2:=\frac{1}{c_0}\psi_1^{-1}(\frac{1}{d_0})\max \left\{\int_0^{\gamma_h}\psi_1^{-1}\left(\int_s^{\gamma_h}h(\tau)d\tau\right)ds, \int_{\gamma_h}^{1}\psi_1^{-1}\left(\int^s_{\gamma_h}h(\tau)d\tau\right)ds\right\}.$
\end{center}
Since $0 \le \alpha_h <\gamma_h^1 <\gamma_h <\gamma_h^2 <\beta_h \le 1,$ $A_1>0$ and $A_2>0.$ Define continuous functions $R_1, R_2 : (0,\infty) \to (0,\infty)$ by
\begin{center}
$\displaystyle R_1(m):=\frac{1}{f_*(m)}\varphi(\frac{m}{A_1})$
and $\displaystyle R_2(m):=\frac{1}{f^*(m)}\varphi(\frac{m}{A_2})$ for $m \in (0,\infty)$.
\end{center}
Here, $f_*(m):=\min\{f(y): {\rho_h m\le y\le m}\}$ and $f^*(m):=\max\{f(y): {0\le y\le m}\}$ for $m \in \mathbb{R}_+$.

\begin{Rem}\label{RemA}
\begin{itemize}
\item[$(1)$] By \eqref{inequality2}, $\psi_2^{-1}(y) \le \psi_1^{-1}(y)$ for all $y \in \mathbb{R}_+,$ and it follows that
\begin{center}
$0<A_1<A_2$ and $0<R_2(m)<R_1(m)$ for all $m \in (0,\infty).$
\end{center}
\item[$(2)$] For any $l \in C(\mathbb{R}_+,\mathbb{R}_+)$,
 $\displaystyle l_a:=\lim_{m\to a}\frac{l(m)}{\varphi(m)}$ for $a \in \{0,\infty\}$.
 It is not hard to prove that
\begin{center}
$(f_*)_a=(f^*)_a=0$ if $f_a=0$ and $(f_*)_a=(f^*)_a=\infty$ if $f_a=\infty$.
\end{center}

For convenience of readers, we give the proofs. First, we show that $f_0=0$ implies $(f_*)_0=(f^*)_0=0$. Let $\epsilon>0$ be given and $f_0=0$ be assumed. Then there exists $\delta>0$ such that  for any $s \in (0,\delta),$ $0<\frac{f(s)}{\varphi(s)} < \epsilon$. Since $f \in C(\mathbb{R}_+,\mathbb{R}_+)$ with $f(s)>0$ for $s \in (0,\infty),$  by the extreme value theorem, for any $s \in (0,\infty),$ $f^*(s)=f(x_s)$ for some $x_s \in (0,s]$. Then
\begin{center}
$\displaystyle 0\le \frac{f_*(s)}{\varphi(s)} \le \frac{f^*(s)}{\varphi(s)}=\frac{f(x_s)}{\varphi(s)} \le \frac{f(x_s)}{\varphi(x_s)}<\epsilon$ for any $s \in (0,\delta),$
\end{center}
which implies that $(f_*)_0=(f^*)_0=0$.

Next, we show that $f_\infty=0$ implies $(f_*)_\infty=(f^*)_\infty=0$. Indeed, let $\epsilon>0$ be given and $f_\infty=0$ be assumed. Then there exists $N>0$ such that for any $s \ge N$,
\begin{center}
$\displaystyle \frac{f(s)}{\varphi(s)}<\epsilon$.
\end{center}
 For $s \ge N$, $f^*(s)\le f^*(N)+f(x_{N,s}),$ where $x_{N,s}$ is the point in $[N,s]$ satisfying $f(x_{N,s})=\max\{f(x): N\le x\le s\}$. Then $$0\le \frac{f_*(s)}{\varphi(s)} \le\frac{f^*(s)}{\varphi(s)}\le  \frac{f^*(N)}{\varphi(s)}+ \frac{f^*(x_{N,s})}{\varphi(x_{N,s})} \le \frac{f^*(N)}{\varphi(s)}+\epsilon.$$
Consequently, $\displaystyle 0\le \limsup_{s \to \infty}\frac{f_*(s)}{\varphi(s)}\le \limsup_{s \to \infty}\frac{f^*(s)}{\varphi(s)}\le \epsilon$, which is true for all $\epsilon>0$. Thus $(f_*)_\infty=(f^*)_\infty=0$.

Finally, we show that, for $a \in \{0,\infty\},$ $f_a=\infty$ implies $(f_*)_a=\infty$. For each $m\in (0,\infty),$ by the extreme value theorem, there exists $m_* \in [\rho_h m,m]$ satisfying $(f_*)(m)=f(m_*),$ and
$$\frac{f_*(m)}{\varphi(m)}=\frac{f(m_*)}{\varphi(m)}\ge \frac{f(m_*)}{\varphi(m_*)\psi_2(\frac{m}{m_*})}\ge \frac{f(m_*)}{\varphi(m_*)\psi_2(\rho_h^{-1})}.$$ As $m \to a\in\{0,\infty\},$ $m_* \to a$, and thus  $(f_*)_a=\infty$, provided $f_a=\infty$.

\item[] $(3)$ By $(A)$ and Remark \ref{RemA} $(2)$, for $i\in \{1,2\}$ and $a \in \{0,\infty\},$
\begin{center}
$\displaystyle  \lim_{m\to a} R_i(m)=\infty$ if $f_a=0$ and $\displaystyle  \lim_{m\to a} R_i(m)=0$ if $f_a=\infty.$
\end{center}

\end{itemize}
\end{Rem}

\begin{Lem}\label{lemA}
Assume that $(A)$ and $h \in \mathcal{H}_{\psi_1}\setminus \{0\}$ hold. Let $m \in (0,\infty)$ be fixed. Then, for any $\lambda \in (R_1(m),\infty)$, $\|H(\lambda,v)\|_\infty > \|v\|$ for all $v \in \partial \mathcal{K}_m$ and
\begin{equation}\label{A1}
i(H(\lambda,\cdot),\mathcal{K}_m,\mathcal{K})=0.
\end{equation}
\end{Lem}
\begin{proof}
Let $\lambda \in (R_1(m),\infty)$ and $v \in \partial \mathcal{K}_m$ be fixed. Then $\rho_h m\le v(t)\le m$ for $t \in [\gamma_h^1,\gamma_h^2],$ and
\begin{equation}\label{A2}
f(v(t))\ge f_*(m)= \frac{1}{R_1(m)}\varphi(\frac{m}{A_1})~\hbox{for}~t \in [\gamma_h^1,\gamma_h^2].
 \end{equation}
Let $\sigma$ be a number satisfying $H(\lambda,v)(\sigma)=\|H(\lambda,v)\|_\infty.$ We have two cases: either $(i)$ $\sigma\in [\gamma_h,1)$ or $(ii)$  $\sigma \in (0,\gamma_h)$. We only consider the case $(i),$ since the case $(ii)$ can be dealt in a similar manner.
Since $\lambda >R_1(m),$ it follows from \eqref{inequality2} and \eqref{A2} that
\begin{eqnarray*}
\|H(\lambda,v)\|_\infty&=&\int_0^{\sigma} \frac{1}{c(s)}\varphi^{-1} \left(\frac{1}{d(s)}\int_s^{\sigma} \lambda h(\tau)f(v(\tau))d\tau\right)ds\\
&\ge& \frac{1}{\|c\|_\infty}\int_{\gamma_h^1}^{\gamma_h}\varphi^{-1}\left(\int_s^{\gamma_h} h(\tau)d\tau\frac{\lambda}{\|d\|_\infty R_1(m)}\varphi(\frac{m}{A_1})\right)ds\\
&\ge&  \frac{1}{\|c\|_\infty}\int_{\gamma_h^1}^{\gamma_h}\psi_2^{-1}\left(\int_s^{\gamma_h} h(\tau)d\tau\right)ds\varphi^{-1}(\frac{\lambda}{\|d\|_\infty R_1(m)}\varphi(\frac{m}{A_1}))\\
&>& \frac{1}{\|c\|_\infty}\int_{\gamma_h^1}^{\gamma_h}\psi_2^{-1}\left(\int_s^{\gamma_h} h(\tau)d\tau\right)ds\psi_2^{-1}(\frac{1}{\|d\|_\infty})\frac{m}{A_1}\ge m=\|v\|_\infty.
\end{eqnarray*}
By Theorem \ref{thm_cone}, \eqref{A1} holds for any $\lambda \in (R_1(m),\infty)$.
\end{proof}

\begin{Lem}\label{lemB}
Assume that $(A)$ and $h \in \mathcal{H}_{\psi_1}\setminus \{0\}$ hold. Let $m \in (0,\infty)$ be fixed. Then, for any $\lambda \in (0,R_2(m)),$ $\|H(\lambda,v)\|_\infty < \|v\|$ for all $v \in \partial \mathcal{K}_m$ and
\begin{equation}\label{B1}
i(H(\lambda,\cdot),\mathcal{K}_m,\mathcal{K})=1.
\end{equation}
\end{Lem}
\begin{proof} Let $\lambda \in (0,R_2(m))$ and $v \in \partial \mathcal{K}_m$ be fixed. Then $f(v(t))\le \frac{1}{R_2(m)}\varphi(\frac{m}{A_2})~\hbox{for}~t \in [0,1].$
By the same argument as in the proof of Lemma \ref{lemA}, it follows that $\|H(\lambda,v)\|_\infty < \|v\|$ for all $v \in \partial \mathcal{K}_m$ and \eqref{B1} holds for any $\lambda \in (0,R_2(m)).$
\end{proof}

By Lemma \ref{lemA} and Lemma \ref{lemB}, we give some results for the existence and multiplicity of positive solutions to problem \eqref{1.1}. Since the proofs are similar, we only give the proof of Theorem \ref{main1}.

\begin{Thm}\label{main1}
Assume that $(A)$ and $h \in \mathcal{H}_{\psi_1}\setminus \{0\}$ hold, and that there exist
$m_1$ and $m_2$ such that $0<m_1<m_2$ (resp., $0<m_2<m_1$) and $R_1(m_1)<R_2(m_2).$
 Then \eqref{1.1} has a positive solution $u=u(\lambda)$ satisfying $m_1<\|u\|_\infty <m_2$ (resp., $m_2<\|u\|_\infty <m_1$) for any $\lambda \in (R_1(m_1),R_2(m_2))$.
\end{Thm}
\begin{proof}
We only prove the case $0<m_1<m_2$, since the other case is similar. Let $\lambda \in (R_1(m_1),R_2(m_2))$ be fixed.
By Lemma \ref{lemA} and Lemma \ref{lemB}, $i(H(\lambda,\cdot),\mathcal K_{m_1},\mathcal{K})=1$ and $i(H(\lambda,\cdot),\mathcal K_{m_2},\mathcal{K})=0.$ Since  $H(\lambda,v)\neq v$ for all $v \in \partial \mathcal{K}_{m_1}$, it follows from the additivity property that $i(H(\lambda,\cdot),\mathcal K_{m_2} \setminus \overline{\mathcal K}_{m_1} ,\mathcal{K})=-1.$ Then there exists $u \in \mathcal K_{m_2} \setminus \overline{\mathcal K}_{m_1}$ such that $H(\lambda,u)=u$ by the solution property. Thus the proof is complete.
\end{proof}

For Theorem \ref{main2} and Theorem \ref{main3}, let $R_*:=\max\{R_1(m_1),R_1(M_1)\}$ and $R^*:=min\{R_2(m_2),R_2(M_2)\}$.
\begin{Thm}\label{main2}
Assume that $(A)$ and $h \in \mathcal{H}_{\psi_1}\setminus \{0\}$ hold, and that there exist
$m_1,m_2$ and $M_1$ (resp., $M_2$) such that $0<m_1<m_2<M_1$ (resp., $0<m_2<m_1<M_2$) and $R_*<R_2(m_2)$ (resp., $R_1(m_1)<R^*$).
 Then \eqref{1.1} has two positive solutions $u_1=u_1(\lambda)$ and $u_2=u_2(\lambda)$ satisfying $m_1<\|u_1\|_\infty <m_2<\|u_2\|_\infty<M_1$ for any $\lambda \in (R_*,R_2(m_2))$ (resp., $m_2<\|u_1\|_\infty <m_1<\|u_2\|_\infty<M_2$ for any $\lambda \in (R_1(m_1),R^*))$.
\end{Thm}

\begin{Thm}\label{main3}
Assume that $(A)$ and $h \in \mathcal{H}_{\psi_1}\setminus \{0\}$ hold, and that there exist
$m_1,m_2,M_1$ and $M_2$ such that $0<m_2<m_1<M_2<M_1$ (resp., $0<m_1<m_2<M_1<M_2$) and $R_* < R^*.$
 Then \eqref{1.1} has three positive solutions $u_1=u_1(\lambda),u_2=u_2(\lambda)$ and $u_3=u_3(\lambda)$ satisfying $m_2<\|u_1\|_\infty <m_1<\|u_2\|_\infty<M_2<\|u_3\|_\infty <M_1$ (resp., $m_1<\|u_1\|_\infty <m_2<\|u_2\|_\infty<M_1<\|u_3\|_\infty <M_2)$ for any $\lambda\in (R_* , R^*)$.
\end{Thm}

Now we give the existence and nonexistence results for positive solutions to problem \eqref{1.1} which are analogous to \cite[Theorem 1.1 and Theorem 1.2]{MR1946560}.

\begin{Thm}\label{main4}
Assume that $(A)$ and $h \in \mathcal{H}_{\psi_1}\setminus \{0\}$ hold.
\begin{itemize}
\item[$(1)$] If $f_0=0$ and $f_\infty=\infty$ (resp., $f_0 =\infty$ and $f_\infty=0$), then \eqref{1.1} has a positive solution $u(\lambda)$ for any $\lambda\in (0,\infty)$ satisfying $\|u_\lambda\|_\infty \to \infty$ as $\lambda \to 0$ and $\|u_\lambda\|_\infty \to 0$ as $\lambda \to \infty$ (resp.,  $\|u_\lambda\|_\infty \to 0$ as $\lambda \to 0$ and $\|u_\lambda\|_\infty \to \infty$ as $\lambda \to \infty$).

\item[$(2)$] If $f_0=0$ and $f_\infty \in (0,\infty)$ (resp., $f_0 \in (0,\infty)$ and $f_\infty=0$), then
 there exists\ $\lambda_*\in (0,\infty)$ and $m_* \in (0,\infty]$ (resp., $m_* \in \mathbb{R}_+)$ such that \eqref{1.1} has a positive solution $u(\lambda)$ for any $\lambda\in (\lambda_*,\infty)$ satisfying $\|u(\lambda)\|_\infty \in (0,m_*)$ and $\displaystyle\lim_{\lambda\to \infty}\|u(\lambda)\|_\infty=0$ (resp., $\|u(\lambda)\|_\infty \in (m_*,\infty)$ and $\displaystyle\lim_{\lambda\to \infty}\|u(\lambda)\|_\infty=\infty)$.
  Moreover, if $m_* \in (0,\infty),$ then there exists a positive solution $u(\lambda_*)$ to problem \eqref{1.1} with $\lambda=\lambda_*$.

\item[$(3)$] If $f_0=\infty$ and $f_\infty \in (0,\infty)$ (resp., $f_0\in (0,\infty)$ and $f_\infty=\infty$), then
 there exist $\lambda^*\in (0,\infty)$ and $m^* \in (0,\infty]$ (resp., $m^*\in \mathbb{R}_+$) such that problem \eqref{1.1} has a positive solution $u=u(\lambda)$ for any $\lambda\in (0,\lambda^*)$ satisfying $\|u(\lambda)\|_\infty \in (0,m^*)$ and $\displaystyle \lim_{\lambda\to 0}\|u(\lambda)\|_\infty= 0$ (resp., $\|u(\lambda)\|_\infty \in (m^*,\infty)$ and $\displaystyle\lim_{\lambda\to 0}\|u(\lambda)\|_\infty = \infty$). Moreover, if $m^* \in (0,\infty),$ then there exists a positive solution $u(\lambda^*)$ to problem \eqref{1.1} with $\lambda=\lambda^*$.

\item[$(4)$] If $f_0=f_\infty=0,$ then there exist $\lambda_* \in (0,\infty)$ and $m_*\in (0,\infty)$ such that \eqref{1.1} has two positive solutions $u_1(\lambda)$ and $u_2(\lambda)$ for any $\lambda\in (\lambda_*,\infty)$ and it has a positive solution $u(\lambda_*)$ for $\lambda=\lambda_*$. Moreover, $u_1(\lambda)$ and $u_2(\lambda)$ can be chosen so that $0<\|u_1(\lambda)\|_\infty<m_*<\|u_2(\lambda)\|_\infty $, $\displaystyle \lim_{\lambda\to \infty}\|u_1(\lambda)\|_\infty= 0$ and $\displaystyle\lim_{\lambda\to \infty}\|u_2(\lambda)\|_\infty = \infty$.

\item[$(5)$] If $f_0=f_\infty=\infty,$ then
 there exist $\lambda^*\in (0,\infty)$ and $m^*\in (0,\infty)$ such that problem \eqref{1.1} has two positive solutions $u_1(\lambda)$ and $u_2(\lambda)$ for any $\lambda\in (0,\lambda^*)$ and it has a positive solution $u(\lambda^*)$ for $\lambda=\lambda^*$. Moreover, $u_1(\lambda)$ and $u_2(\lambda)$ can be chosen so that $0<\|u_1(\lambda)\|_\infty<m^*<\|u_2(\lambda)\|_\infty $, $\displaystyle \lim_{\lambda\to 0}\|u_1(\lambda)\|_\infty= 0$ and $\displaystyle\lim_{\lambda\to 0}\|u_2(\lambda)\|_\infty = \infty$.
\item[$(6)$] If $f_0 \in [0,\infty)$ and $f_\infty \in [0,\infty)$, then there exists $\overline \lambda>0$ such that \eqref{1.1} has no positive solutions for $\lambda <\overline \lambda.$
\item[$(7)$] If $f_0 \in (0,\infty]$ and $f_\infty \in (0,\infty]$, then there exists $\underline \lambda>0$ such that \eqref{1.1} has no positive solutions for $\lambda > \underline \lambda.$
\end{itemize}
\end{Thm}
\begin{proof}  $(1)$ We only give the proof of the case $f_0=0$ and $f_\infty=\infty,$ since the case $f_0=\infty$ and $f_\infty=0$ can be proved in a similar manner. Since $f_0=0$ and $f_\infty=\infty,$ by Remark \ref{RemA} (3), $R_i(m) \to \infty$ as $m \to 0$ and $R_i(m) \to 0$ as $m \to \infty$ for $i=1,2.$ For any $\lambda \in (0,\infty),$ there exist $m_1(\lambda)$ and $m_2(\lambda)$ such that
\begin{center}
$0<m_2(\lambda)<m_1(\lambda)$ and $R_1(m_1(\lambda))<\lambda < R_2(m_2(\lambda))$.
\end{center}
 By Theorem \ref{main1}, there exists a positive solution $u_\lambda$ to problem \eqref{1.1} satisfying
  \begin{center}
  $m_2(\lambda)<\|u_\lambda\|_\infty <m_1(\lambda).$
\end{center}
Since $R_i(m) \to \infty$ as $m \to 0$ for $i=1,2$, we may choose $m_1(\lambda)$ and $m_2(\lambda)$ so that $0<m_2(\lambda)<m_1(\lambda)$ and $m_1(\lambda) \to 0$ as $\lambda \to \infty$. Consequently, we can choose positive solutions $u_\lambda$ to problem \eqref{1.1} for large $\lambda>0$ so that $\|u_\lambda\|_\infty \to 0$ as $\lambda \to \infty$. Similarly, since $R_i(m) \to 0$ as $m \to \infty$ for $i=1,2$, we can choose positive solutions $u_\lambda$ to problem \eqref{1.1} for small $\lambda>0$ so that $\|u_\lambda\|_\infty \to \infty$ as $\lambda \to 0$.

 $(2)$ We only give the proof of the case $f_0=0$ and $f_\infty \in (0,\infty)$, since the case $f_0\in (0,\infty)$ and $f_\infty=0$ can be proved in a similar manner. Since $f_0=0$, by Remark \ref{RemA} (3), $R_i(m) \to \infty$ as $m \to 0$ for $i=1,2.$ Since
 \begin{center}
 $\displaystyle \lim_{m \to \infty} R_1(m)\ge \lim_{m \to \infty} \frac{\varphi(\frac{m}{A_1})}{f(m)} \ge  \lim_{m \to \infty} \frac{\varphi(m)}{f(m)}\psi_1(\frac{1}{A_1})  = \frac{1}{f_\infty}\psi_1(\frac{1}{A_1})>0,$
 \end{center}
there exist $\lambda_*:=\inf\{R_1(m): m \in (0,\infty)\}\in (0,\infty)$ and $m_* \in (0,\infty]$ satisfying $R_1(m_*)=\lambda_*.$  For any $\lambda \in (\lambda_*,\infty),$ there exist $m_1(\lambda)$ and $m_2(\lambda)$ such that
\begin{center}
$0<m_2(\lambda)<m_1(\lambda)<m_*$ and $R_1(m_1)<\lambda < R_2(m_2)$.
\end{center}
 By Theorem \ref{main1}, there exists a positive solution $u_\lambda$ to problem \eqref{1.1} satisfying $m_2(\lambda)<\|u\|_\infty <m_1(\lambda).$ Since $R_i(m) \to \infty$ as $m \to 0$ for $i=1,2,$ we may choose $m_1(\lambda)$ and $m_2(\lambda)$ satisfying
  \begin{center}
  $0<m_2(\lambda)<m_1(\lambda)$ and $m_1(\lambda) \to 0$ as $\lambda \to \infty$.
\end{center}
Consequently,  we can choose positive solutions $u_\lambda$ to problem \eqref{1.1} for large $\lambda>0$ so that $\|u_\lambda\|_\infty \to 0$ as $\lambda \to \infty$.

For each $n \in \mathbb{N}$, let $\lambda_n:=\lambda_*+\frac{1}{n}$. Then we may choose $m_1=m_1(n)$ and $m_2=m_2(n)$ such that
\begin{center}
$R_1(m_1(n))<\lambda_n<R_2(m_2(n))$ and
$0<\delta< m_2(n)< m_1(n)<m_*$ for all $n.$
\end{center}
Consequently, for each $n,$ there exists $u_n\in \mathcal{K}$ such that $H(\lambda_n,u_n)=u_n$ and $\delta<\|u_n\|_\infty < m_*$.  Since $\{(\lambda_n,u_n)\}$ is bounded in $\mathbb{R}_+\times \mathcal{K}$ and $H:\mathbb{R}_+\times \mathcal{K} \to \mathcal{K}$ is compact, there exist a subsequence $\{(\lambda_{n_k},u_{n_k})\}$ of $\{(\lambda_n,u_n)\}$ and $u_*\in \mathcal{K}$ such that $H(\lambda_{n_k},u_{n_k})=u_{n_k} \to u_*$ in $\mathcal{K}$ as $k \to \infty.$ Since $\lambda_n \to \lambda_*$ as $n\to \infty$ and $H$ is continuous, $H(\lambda_*,u_*)=u_*$ and $\|u_*\|_\infty\ge \delta>0.$ Thus problem \eqref{1.1} has a positive solution $u_*$ for $\lambda=\lambda_*$. Thus the proof is complete.

$(3)$ Let $\lambda^*=\sup\{R_2(m): m \in (0,\infty)\}\in (0,\infty)$ and $m^* \in [0,\infty]$ satisfying $R_2(m^*)=\lambda^*.$ Then the proof is complete by the argument similar to those in the proof of Theorem \ref{main4} $(2).$

$(4)$ Since $f_0=f_\infty=0,$ it follows that, for $i=1,2,$ $\displaystyle  \lim_{m\to 0} R_i(m)=\lim_{m\to \infty} R_i(m)=\infty$. Then there exists $m_* \in (0,\infty)$ satisfying
\begin{center}
$R_1(m_*)=\min \{R_1(m): m \in \mathbb{R}_+\} \in (0,\infty).$
\end{center}
 Let $\lambda_*=R_1(m_*).$ For any $\lambda \in (\lambda_*,\infty)$, there exist $m_1(\lambda),m_2(\lambda),M_1(\lambda)$ and $M_2(\lambda)$ such that
 \begin{center}
  $0<m_2(\lambda)<m_1(\lambda)<m_* <M_1(\lambda)<M_2(\lambda)$
  \end{center}
   and
   \begin{center}
$R_1(m_1(\lambda))=R_1(M_1(\lambda)) <\lambda < R_2(m_2(\lambda))=R_2(M_2(\lambda)).$
\end{center}
Then the proof is complete by the argument similar to those in the proof of Theorem \ref{main4} $(2).$

 $(5)$ Since $f_0=f_\infty=\infty,$ it follows that, for $i=1,2,$ $\displaystyle  \lim_{m\to 0} R_i(m)=\lim_{m\to \infty} R_i(m)=0$. Let $\lambda^*=\max\{R_2(m): m \in \mathbb{R}_+\}\in (0,\infty)$ and $m^* \in (0,\infty)$ satisfying $R_2(m^*)=\lambda^*.$ Then the proof is complete by the argument similar to those in the proof of Theorem \ref{main4} $(2).$

$(6)$ Let $u$ be a positive solution to problem \eqref{1.1} with $\lambda>0$ and let $\sigma$ be a constant satisfying $u(\sigma)=\|u\|_\infty$. Since $f_0 \in [0,\infty)$ and $f_\infty \in [0,\infty)$, there exists $C_1>0$ such that $f(s) \le C_1 \varphi(s)$ for $s \in\mathbb{R}_+.$ We only consider the case $\sigma \le \gamma_h,$ since the case $\sigma >\gamma_h$ can be dealt in a similar manner.
Since $f(u(t)) \le C_1 \varphi(u(\sigma))$ for $t \in [0,1],$
\begin{eqnarray*}
u(\sigma)&=&\int_0^{\sigma}\frac{1}{c(s)}\varphi^{-1} \left(\frac{1}{d(s)}\int_s^\sigma \lambda h(\tau)f(u(\tau))d\tau\right)ds\\
&\le& c_0^{-1}\int_0^{\gamma_h} \varphi^{-1} \left(\int_{s}^{\gamma} h(\tau)d\tau d_0^{-1}\lambda  C_1\varphi(u(\sigma))\right)ds\\
&\le& c_0^{-1}h^* \varphi^{-1}(d_0^{-1}\lambda  C_1 \varphi(u(\sigma)))\le c_0^{-1}h^* \psi_1^{-1}(d_0^{-1}\lambda  C_1) u(\sigma).
\end{eqnarray*}
Here $$h^*=\max \left\{\int_0^{\gamma_h}\psi_1^{-1}\left(\int_s^{\gamma_h}h(\tau)d\tau\right)ds, \int_{\gamma_h}^{1}\psi_1^{-1}\left(\int^s_{\gamma_h}h(\tau)d\tau\right)ds\right\}>0.$$
Consequently,
\begin{center}
$\displaystyle\lambda \ge \frac{d_0}{C_1}\psi_1(\frac{c_0}{h^*})=:\overline\lambda.$
\end{center}

$(7)$ Let $u$ be a positive solution to problem \eqref{1.1} with $\lambda>0$ and let $\sigma$ be a constant satisfying $u(\sigma)=\|u\|_\infty$. Since $f_0 \in (0,\infty]$ and $f_\infty \in (0,\infty]$, there exists $\epsilon>0$ such that $f(s)> \epsilon \varphi(s)$ for $s \in\mathbb{R}_+.$ We only consider the case $\sigma \ge \gamma_h,$ since the case $\sigma <\gamma_h$ can be dealt in a similar manner.
Since $u(t) \ge u(\gamma_h^1)$ for  $t \in [\gamma_h^1,\sigma],$
$f(u(t))> \epsilon \varphi(u(\gamma_h^1))$ for $t \in [\gamma_h^1,\gamma].$ Then
\begin{eqnarray*}
u(\gamma_h^1)&=&\int_0^{\gamma_h^1}\frac{1}{c(s)}\varphi^{-1} \left(\frac{1}{d(s)}\int_s^\sigma \lambda h(\tau)f(u(\tau))d\tau\right)ds\\
&\ge& \|c\|_\infty^{-1}\int_0^{\gamma_h^1} \varphi^{-1} \left(\int_{\gamma_h^1}^{\gamma_h} h(\tau)d\tau \|d\|_\infty^{-1}\lambda  \epsilon \varphi(u(\gamma_h^1))\right)ds\\
&\ge&\|c\|_\infty^{-1}\gamma_0 \varphi^{-1}(h_* \|d\|_\infty^{-1}\lambda  \epsilon \varphi(u(\gamma_h^1)))
\ge \|c\|_\infty^{-1}\gamma_0 \psi_2^{-1}(h_* \|d\|_\infty^{-1}\lambda  \epsilon) u(\gamma_h^1).
\end{eqnarray*}
Here
\begin{center}
$\displaystyle \gamma_0=\min\{\gamma_h^1, 1-\gamma_h^2\}>0$ and $\displaystyle h_* = \min\{\int_{\gamma_h^1}^{\gamma_h} h(\tau)d\tau, \int_{\gamma_h}^{\gamma_h^2} h(\tau)d\tau\}>0.$
\end{center}
Consequently,
\begin{center}
$\displaystyle \lambda \le \frac{\|d\|_\infty}{h_* \epsilon}\psi_2(\frac{\|c\|_\infty}{\gamma_0})=:\underline \lambda.$
\end{center}
\end{proof}

Finally, an example to illustrate the results obtained in this paper is given.
\begin{Ex}
Let $\varphi$ be an odd function satisfying either
\begin{center}
$(i)$ $\varphi(x)=\frac{x^2}{1+x}$ or $(ii)$ $\varphi(x)=x+x^2$ for $x \ge 0$.
\end{center}
 Then it is easy to check that  $(A)$ is satisfied for
\begin{center}
$\psi_1(y)=\min\{y,y^2\}$ and $\psi_2(y)=\max\{y,y^2\}$.
\end{center}
 Define $h:(0,1) \to \mathbb{R}_+$ by
 \begin{center}
 $h(t)=0$ for $t \in [0,\frac{1}{16}]$ and $h(t)=(t-\frac{1}{16})(1-t)^{-a}$ for $t \in (\frac{1}{16},1).$
\end{center}
Then, since $\psi_1^{-1}(s)=s$ for all $s \ge 1,$ $h \in \mathcal{H}_{\psi_1} \setminus L^1(0,1)$ for any $a\in [1,2)$. Also, $\alpha_h=\bar \beta_h=\frac{1}{16}$, $\beta_h=\bar \alpha_h=1,$ $\gamma_h^1=\frac{19}{64},$ $\gamma_h^2=\frac{49}{64}$ and $\gamma_h=\frac{17}{32}.$

Let $c,d$ be any positive continuous functions on $[0,1]$. Then $\rho_1,$ $\rho_h,$ $A_1$ and $A_2$ are well defined in $(0,\infty).$ Define $f :\mathbb{R}_+ \to \mathbb{R}_+$ by
$$f(s)=\left\{
                       \begin{array}{ll}
                       (\varphi(s))^2, & \hbox{for}~s\in[0,1],\\
                         (\varphi(1))^{\frac{3}{2}}(\varphi(s))^{\frac{1}{2}}, & \hbox{for} s \in (1,M_2], \\
                        (\varphi(1))^{\frac{3}{2}} (\varphi(M_2))^{-\frac{3}{2}}(\varphi(s))^2, & \hbox{for} s \in (M_2,\infty).
                       \end{array}
                     \right.$$
Here $M_2$ is a fixed constant satisfying
\begin{center}
$\displaystyle M_2 > \max\left\{\frac{1}{\rho_h}, \varphi^{-1}\left(\frac{1}{\varphi(1)}\left[\varphi\left(\frac{1}{\rho_h A_1}\right)\right]^2\left[\psi_1\left(\frac{1}{A_2}\right)\right]^{-2}\right)\right\}.$
\end{center}
Clearly, $f \in C(\mathbb{R}_+,\mathbb{R}_+)$ satisfies that $f(s)>0$ for $s >0,$  $f_0=0$ and $f_\infty=\infty.$
Since $f$ is strictly increasing on $\mathbb{R}_+,$ $f^*(m)=f(m)$ and $f_*(m)=f(\rho_h m)$ for all $m \in \mathbb{R}_+.$ Then, by \eqref{inequality2},
\begin{center}
$ \displaystyle R_1(m)= \frac{1}{f(\rho_h m)}\varphi(\frac{m}{A_1})$
 and
$\displaystyle R_2(m)=\frac{1}{f(m)}\varphi(\frac{m}{A_2}) \ge \frac{\varphi(m)}{f(m)}\psi_1(\frac{1}{A_2})$ for $m \in (0,\infty).$
\end{center}
 From the choice of $M_2$, it follows that
$$R_1(\frac{1}{\rho_h}) = \frac{1}{f(1)} \varphi(\frac{1}{\rho_h A_1})= \frac{1}{[\varphi(1)]^2} \varphi(\frac{1}{\rho_h A_1})<\frac{[\varphi(M_2)]^{\frac{1}{2}}}{[\varphi(1)]^{\frac{3}{2}}}\psi_1(\frac{1}{A_2})= \frac{\varphi(M_2)}{f(M_2)}\psi_1(\frac{1}{A_2}) \le R_2(M_2).$$
We may choose $m_2$ and $M_1$ satisfying $0<m_2<m_1=\rho_h^{-1}<M_2<M_1$ and
 \begin{center}
$R_1(M_1)\le R_1(\rho_h^{-1})<R_2(M_2) \le R_2(m_2),$
\end{center}
since $R_2(m) \to \infty$ as $m \to 0$ and $R_1(m) \to 0$ as $m \to \infty$. By Theorem \ref{main3} and Theorem \ref{main4} $(1)$, problem \eqref{1.1} has three positive solutions for $\lambda \in (R_1(\rho_h^{-1}),R_2(M_2))$ and it has a positive solution $u(\lambda)$ for $\lambda \in (0,\infty)$ satisfying that $\displaystyle \lim_{\lambda\to 0}\|u(\lambda)\|_\infty=\infty$ and $\displaystyle\lim_{\lambda\to \infty}\|u(\lambda)\|_\infty=0$.
\end{Ex}

\section{Conclusions}

In this work, we studied the existence and nonexistence of positive solutions to problem \eqref{1.1}. In Theorem \ref{main1}, Theorem \ref{main2}, Theorem \ref{main3} and Theorem \ref{main4}, various sufficient conditions on the nonlinearity $f$ for the existence, nonexistence and multiplicity of positive solutions to problem \eqref{1.1} were given. In particular, Theorem \ref{main4} improves on the results in \cite{MR1946560}, since we do not assume the monotonicity of $d$, and the weight function $h: (0,1) \to \mathbb{R}_+$ may not be $L^1(0,1)$ and it can be vanished in some subinterval of $(0,1)$.

In \cite{jeong2019existence}, the nonlinearity $f=f(t,s)$ should satisfy the condition
\begin{equation}\label{condition}
f(t_0,0)>0~ \hbox{for some}~t_0\in [0,1],
\end{equation}
so that the existence of an unbounded solution component was shown and, by examining the shape of the component according to several assumptions on $f$ at $\infty$, the existence, nonexistence and multiplicity of positive solutions were studied. In this case, all nonnegative solutions are positive ones by \eqref{condition}. Compared with the results in \cite{jeong2019existence}, the nonlinearity $f=f(s)$ in the present work may have the property $f(0)=0$. Even though the existence of unbounded solution component to problem \eqref{1.1} can be obtained by \cite[Theorem 1]{jeong2019existence}, the existence of positive solutions cannot be shown from the solution component, since problem \eqref{1.1} has a trivial solution $0$ for every $\lambda \in \mathbb{R}_+$, provided $f(0)=0.$ Thus, the fixed point index theory was used in order to prove the main results (Theorem \ref{main1}, Theorem \ref{main2}, Theorem \ref{main3} and Theorem \ref{main4}).

In the present work, the problem with Dirichlet boundary conditions was considered. As an extension of the results in this paper, similar results for the problem with nonlocal boundary conditions is expected. The existence, nonexistence and multiplicity of positive solutions to problem with nonlocal boundary conditions will be discussed for future work.

{\bf Acknowledgements}  This research was supported by Basic Science Research Program through the National Research Foundation of Korea(NRF) funded by the Ministry of Education(2017R1D1A1B03035623).


\begin{thebibliography}{10}


\bibitem{jeong2019existence}
Jeong, J.; Kim, C.G.
\newblock Existence of Positive Solutions to Singular Boundary Value Problems
  Involving $\varphi$-Laplacian.
\newblock {\em Mathematics} {\bf 2019}, {\em 7},~654.


\bibitem{MR1946560}
Wang, H.
\newblock On the structure of positive radial solutions for quasilinear
  equations in annular domains.
\newblock {\em Adv. Differential Equations} {\bf 2003}, {\em 8},~111--128.

\bibitem{MR903391}
Bandle, C.; Coffman, C.V.; Marcus, M.
\newblock Nonlinear elliptic problems in annular domains.
\newblock {\em J. Differential Equations} {\bf 1987}, {\em 69},~322--345.


\bibitem{MR1894171}
Graef, J.R.; Yang, B.
\newblock Boundary value problems for second order nonlinear ordinary
  differential equations.
\newblock {\em Commun. Appl. Anal.} {\bf 2002}, {\em 6},~273--288.

\bibitem{MR1440355}
Henderson, J.; Wang, H.
\newblock Positive solutions for nonlinear eigenvalue problems.
\newblock {\em J. Math. Anal. Appl.} {\bf 1997}, {\em 208},~252--259.


\bibitem{MR3231529}
Iturriaga, L.; Massa, E.; S\'{a}nchez, J.; Ubilla, P.
\newblock Positive solutions for an elliptic equation in an annulus with a
  superlinear nonlinearity with zeros.
\newblock {\em Math. Nachr.} {\bf 2014}, {\em 287},~1131--1141.


\bibitem{MR1825983}
Lan, K.Q.
\newblock Multiple positive solutions of semilinear differential equations with
  singularities.
\newblock {\em J. London Math. Soc. (2)} {\bf 2001}, {\em 63},~690--704.


\bibitem{MR1643199}
Lan, K.; Webb, J.R.L.
\newblock Positive solutions of semilinear differential equations with
  singularities.
\newblock {\em J. Differential Equations} {\bf 1998}, {\em 148},~407--421.


\bibitem{MR1064016}
Lin, S.S.
\newblock Positive radial solutions and nonradial bifurcation for semilinear
  elliptic equations in annular domains.
\newblock {\em J. Differential Equations} {\bf 1990}, {\em 86},~367--391.

\bibitem{MR1880513}
Agarwal, R.P.; L\"u, H.; O'Regan, D.
\newblock Eigenvalues and the one-dimensional {$p$}-{L}aplacian.
\newblock {\em J. Math. Anal. Appl.} {\bf 2002}, {\em 266},~383--400.


\bibitem{MR1003248}
del Pino, M.; Elgueta, M.; Man\'{a}sevich, R.
\newblock A homotopic deformation along {$p$} of a {L}eray-{S}chauder degree
  result and existence for {$(|u'|^{p-2}u')'+f(t,u)=0,\;u(0)=u(T)=0,\;p>1$}.
\newblock {\em J. Differential Equations} {\bf 1989}, {\em 80},~1--13.


\bibitem{MR2514757}
Kim, C.G.
\newblock Existence of positive solutions for singular boundary value problems
  involving the one-dimensional {$p$}-{L}aplacian.
\newblock {\em Nonlinear Anal.} {\bf 2009}, {\em 70},~4259--4267.


\bibitem{MR2847467}
Kim, C.G.
\newblock The three-solutions theorem for {$p$}-{L}aplacian boundary value
  problems.
\newblock {\em Nonlinear Anal.} {\bf 2012}, {\em 75},~924--931.


\bibitem{MR3911665}
Rynne, B.P.
\newblock Exact multiplicity and stability of solutions of a 1-dimensional,
  {$p$}-{L}aplacian problem with positive convex nonlinearity.
\newblock {\em Nonlinear Anal.} {\bf 2019}, {\em 183},~271--283.


\bibitem{MR3361694}
Sim, I.; Tanaka, S.
\newblock Three positive solutions for one-dimensional {$p$}-{L}aplacian
  problem with sign-changing weight.
\newblock {\em Appl. Math. Lett.} {\bf 2015}, {\em 49},~42--50.


\bibitem{MR1423340}
Wang, J.
\newblock The existence of positive solutions for the one-dimensional
  {$p$}-{L}aplacian.
\newblock {\em Proc. Amer. Math. Soc.} {\bf 1997}, {\em 125},~2275--2283.


\bibitem{MR1400567}
Wang, J.; Gao, W.
\newblock A singular boundary value problem for the one-dimensional
  {$p$}-{L}aplacian.
\newblock {\em J. Math. Anal. Appl.} {\bf 1996}, {\em 201},~851--866.


\bibitem{yang2018positive}
Yang, G.; Li, Z.
\newblock Positive Solutions of One-Dimensional $p$-Laplacian Problems with
  Superlinearity.
\newblock {\em Symmetry} {\bf 2018}, {\em 10},~363.

\bibitem{MR3543777}
Shivaji, R.; Sim, I.; Son, B.
\newblock A uniqueness result for a semipositone {$p$}-{L}aplacian problem on
  the exterior of a ball.
\newblock {\em J. Math. Anal. Appl.} {\bf 2017}, {\em 445},~459--475.


\bibitem{MR3959306}
Xu, X.
\newblock A new existence result for the boundary value problem of
  {$p$}-{L}aplacian equations with sign-changing weights.
\newblock {\em Proc. Indian Acad. Sci. Math. Sci.} {\bf 2019}, {\em 129},~Art.
  47, 8.


\bibitem{MR3001526}
Bai, D.; Chen, Y.
\newblock Three positive solutions for a generalized {L}aplacian boundary value
  problem with a parameter.
\newblock {\em Appl. Math. Comput.} {\bf 2013}, {\em 219},~4782--4788.


\bibitem{not-published}
Lee, Y.H.; Xu, X.
\newblock Existence and Multiplicity Results for Generalized Laplacian Problems
  with a Parameter.
\newblock {\em Bull. Malays. Math. Sci. Soc.} {\bf 2018}.


\bibitem{MR3759527}
Kaufmann, U.; Milne, L.
\newblock Positive solutions for nonlinear problems involving the
  one-dimensional {$\varphi$}-{L}aplacian.
\newblock {\em J. Math. Anal. Appl.} {\bf 2018}, {\em 461},~24--37.


\bibitem{MR2393405}
Sim, I.
\newblock On the existence of nodal solutions for singular one-dimensional
  {$\phi$}-{L}aplacian problem with asymptotic condition.
\newblock {\em Commun. Pure Appl. Anal.} {\bf 2008}, {\em 7},~905--923.


\bibitem{MR2502700}
Garc\'{\i}a-Huidobro, M.; Man\'{a}sevich, R.; Ward, J.R.
\newblock Positive solutions for equations and systems with
  {$p$}-{L}aplace-like operators.
\newblock {\em Adv. Differential Equations} {\bf 2009}, {\em 14},~401--432.

\bibitem{MR3850394}
Kaufmann, U.; Milne, L.
\newblock On one-dimensional superlinear indefinite problems involving the
  {$\phi$}-{L}aplacian.
\newblock {\em J. Fixed Point Theory Appl.} {\bf 2018}, {\em 20},~Art. 134, 9.


\bibitem{MR2417920}
Kim, C.G.; Lee, Y.H.
\newblock Existence of multiple positive solutions for {$p$}-{L}aplacian
  problems with a general indefinite weight.
\newblock {\em Commun. Contemp. Math.} {\bf 2008}, {\em 10},~337--362.


\bibitem{MR787404}
Deimling, K.
\newblock {\em Nonlinear functional analysis}; Springer-Verlag, Berlin,  1985.

\bibitem{MR959889}
Guo, D.J.; Lakshmikantham, V.
\newblock {\em Nonlinear problems in abstract cones}; Vol.~5, {\em Notes and
  Reports in Mathematics in Science and Engineering}, Academic Press, Inc.,
  Boston, MA,  1988.

\end{thebibliography}
\end{document}